\documentclass[11pt,a4paper]{amsart}
\usepackage[T1]{fontenc}
\usepackage{lmodern}
\usepackage[margin=28mm,headheight=14pt]{geometry}
\usepackage{amsmath,amssymb,amsthm,mathtools,mathrsfs}
\usepackage{booktabs,longtable,array,enumitem,microtype}
\usepackage[dvipsnames]{xcolor}
\usepackage{xurl}
\usepackage[colorlinks=true,linkcolor=MidnightBlue,citecolor=MidnightBlue,urlcolor=MidnightBlue]{hyperref}
\usepackage{bookmark}

\setlist{itemsep=0.2em,topsep=0.4em}

\newtheorem{theorem}{Theorem}[section]
\newtheorem{proposition}[theorem]{Proposition}
\newtheorem{lemma}[theorem]{Lemma}

\theoremstyle{definition}
\newtheorem{definition}[theorem]{Definition}

\theoremstyle{remark}
\newtheorem{remark}[theorem]{Remark}
\numberwithin{equation}{section}

\DeclareMathOperator{\remn}{rem}
\newcolumntype{P}[1]{>{\raggedright\arraybackslash}p{#1}}

\title{Rational Points on Drinfeld Modular Curves of Prime Level}

\author{Ryotaro Sakamoto}

\address{Department of Mathematics\\University of Tsukuba\\1-1-1 Tennodai\\Tsukuba\\Ibaraki 305-8571\\Japan}
\email{rsakamoto@math.tsukuba.ac.jp}

\date{\today}


\begin{document}

\begin{abstract}
Let $q\ge3$ be a prime power and let $\mathfrak p\subset\mathbf F_q[T]$
be a nonzero prime ideal of degree $d$. We prove that
$Y_0(\mathfrak p)(\mathbf F_q(T))=\emptyset$ if $d \ge 4$ is even or $d>q$,
and that $Y_1(\mathfrak p)(\mathbf F_q(T))=\emptyset$ for every $d\ge3$.
\end{abstract}

\maketitle

\section{Introduction}

The arithmetic of Drinfeld modules over global function fields
shares many features with that of elliptic curves over number fields.
In particular, the structure of their rational torsion submodules has attracted considerable interest.  
Poonen conjectured uniform boundedness of rational torsion
for Drinfeld modules with fixed coefficient ring, rank,
and degree of the field of definition.
He proved the conjecture in rank one and established
uniform boundedness of the $\mathfrak{p}$-primary rational torsion
of rank-two Drinfeld modules over a fixed finite extension
of $\mathbf F_q(T)$, for each fixed nonzero prime ideal $\mathfrak{p}$ of $\mathbf F_q[T]$ \cite{Poo97}. 
Here $q$ denotes a prime power. 
Schweizer subsequently investigated uniform bounds
for the torsion of rank-two Drinfeld modules \cite{Sch03}.
More precisely, Schweizer conjectured that, for every
rank-two Drinfeld $\mathbf F_q[T]$-module $\phi$ over $\mathbf F_q(T)$, if
\[
\phi( \mathbf{F}_q(T))_{\mathrm{tors}}
\cong \mathbf{F}_q[T]/\mathfrak{a}\oplus \mathbf{F}_q[T]/\mathfrak{b},
\qquad \mathfrak{a} \mid \mathfrak{b},
\]
then
\[
\deg(\mathfrak a) + \deg(\mathfrak b) \le 2.
\]
In particular, Schweizer's conjecture implies $Y_1(\mathfrak p)(\mathbf{F}_q(T))=\emptyset$ for any nonzero prime ideal $\mathfrak{p} \subset \mathbf{F}_q[T]$ with $\deg(\mathfrak{p}) \geq 3$; see also \cite[Conjecture~1.2]{Arm12}. Here, $Y_1(\mathfrak p)$ denotes the Drinfeld modular curve parametrizing rank-two Drinfeld $\mathbf{F}_q[T]$-modules equipped with a torsion point whose annihilator is $\mathfrak p$.

A powerful approach to studying rational torsion submodules of Drinfeld modules is to use the geometry of Drinfeld modular curves.  
This approach has its origins in Mazur's work on rational isogenies of elliptic curves.
Mazur introduced the formal immersion method at a cusp
of a modular curve over $\mathbf Q$, and combined it with the finiteness of the Mordell--Weil group of the Eisenstein quotient of its Jacobian \cite{Maz77,Maz78}. 

Let $\mathfrak p$ be a nonzero prime ideal of $A=\mathbf{F}_q[T]$, and let $Y_0(\mathfrak p)$ denote the coarse moduli curve parametrizing rank-two Drinfeld $A$-modules equipped with a cyclic $\mathfrak p$-subgroup. 
We write $X_0(\mathfrak p)$ for the compactification of $Y_0(\mathfrak p)$. 
In the function field setting, Tamagawa constructed the Eisenstein quotient of $\operatorname{Jac}(X_0(\mathfrak p))$ and proved that its minus quotient with respect to the Atkin--Lehner involution has finitely many $\mathbf{F}_q(T)$-rational points \cite[Theorem~5.7(i)]{Tam95}. 
P\'al subsequently developed the theory of the Eisenstein ideal
\cite{Pal07} and proved that the cuspidal divisor group
maps isomorphically onto the group of $\mathbf{F}_q(T)$-rational points
of the Eisenstein quotient \cite[Theorem~3.5]{Pal10}.
Moreover, for $q=2$, P\'al also proved the following results in  {\cite[Theorems~1.2 and~1.4]{Pal10}}.

\begin{theorem}[P\'al]\label{thm:pal}
One has  $Y_0(\mathfrak p)(\mathbf{F}_2(T))=\emptyset$ for any nonzero prime ideal $\mathfrak{p} \subset \mathbf{F}_2[T]$ with $\deg(\mathfrak{p}) \geq 3$. 
Moreover, Schweizer's torsion conjecture holds over
$\mathbf F_2(T)$.
\end{theorem}

P\'al's proof of the nonexistence of rational points uses
the winding quotient of $\operatorname{Jac}(X_0(\mathfrak p))$
and a variant of Mazur's formal immersion method at the
infinite place, rather than at finite places as in
Mazur's original argument.

Armana used Drinfeld modular forms and the winding quotient
of $\operatorname{Jac}(X_0(\mathfrak p))$ to study rational points
on $Y_1(\mathfrak p)$.

\begin{theorem}[Armana {\cite[Th\'eor\`eme~1.4]{Arm12}}]
For any prime power $q$, the following statements hold.
\begin{itemize}
\item[(1)] If $\deg(\mathfrak{p})=3$, then $Y_1(\mathfrak p)$ has no rational
points over any extension of $\mathbf{F}_q(T)$ of degree at most two.
\item[(2)] If $\deg(\mathfrak{p})=4$, then $Y_1(\mathfrak p)(\mathbf{F}_q(T))=\emptyset$. 
\end{itemize}
\end{theorem}

Armana further established a general nonexistence criterion
under additional hypotheses involving Hecke operators
and Drinfeld modular forms
\cite[Th\'eor\`eme~1.5]{Arm12}.

Ishii subsequently investigated the isogeny characters of rank-two Drinfeld modules and used conditions arising from potentially good reduction at finite places to derive restrictions on the existence of rational cyclic isogenies. 

\begin{theorem}[Ishii {\cite[Corollary~5.3]{Ish22}}]\label{thm:ishii}
For any prime power $q$, if $\deg(\mathfrak{p})=4$, then $Y_0(\mathfrak p)(\mathbf{F}_q(T))=\emptyset$. 
\end{theorem}

More recently, Ishii obtained the following result in his unpublished note.

\begin{theorem}[Ishii {\cite[Corollary~1.2]{IshiiNote}}]
Suppose that $q \geq 5$  is a prime number.
If $\deg(\mathfrak{p})=3$, then $Y_0(\mathfrak p)(\mathbf{F}_q(T))=\emptyset$. 
\end{theorem}

As the preceding results show, for $q>2$, general results on the nonexistence of $\mathbf{F}_q(T)$-rational points on $Y_0(\mathfrak p)$ and $Y_1(\mathfrak p)$
were previously available mainly for small values of $d$.
In this paper, we extend the formal immersion method to prove the nonexistence of $\mathbf{F}_q(T)$-rational points on $Y_0(\mathfrak p)$ and $Y_1(\mathfrak p)$ in a substantially more general setting.

\begin{theorem}[{Theorem \ref{thm:even}, \ref{thm:main-y0}, and \ref{thm:rational-torsion}}]\label{thm:main}
Suppose that $q \neq 2$. (Note that $\mathbf{F}_q(T)$ may still have characteristic $2$.)   
For any nonzero prime ideal $\mathfrak{p} \subset \mathbf{F}_q[T]$, the following assertions hold.
\begin{enumerate}[label=(\arabic*)]
\item If $\deg(\mathfrak{p}) \ge 4$ is even or $d>q$, then $Y_0(\mathfrak p)(\mathbf{F}_q(T))=\emptyset$.
\item If $\deg(\mathfrak{p}) \ge 3$, then $Y_1(\mathfrak p)(\mathbf{F}_q(T))=\emptyset$.
\end{enumerate}
\end{theorem}

\begin{remark}
The case $\deg(\mathfrak{p})=4$ of Theorem~\ref{thm:main}(1) follows from
Ishii's result stated in Theorem~\ref{thm:ishii}. 
The proof given in the present paper does not establish the case $\deg(\mathfrak{p})=4$. 
\end{remark}

\begin{remark}
    When $q=2$, both assertions (1) and (2) of Theorem~\ref{thm:main}
follow from P\'al's theorem stated in Theorem \ref{thm:pal}.
Although many of the arguments in this paper remain valid for $q=2$, the assumption $q\ge 3$ is needed to establish certain inequalities. 
\end{remark}

We now briefly explain a key result in the proof of Theorem~\ref{thm:main}.
The usual formal immersion method relies on the nonvanishing of differentials at the cusps. 
In the function field settig, however, the relevant differentials may vanish at the cusps. 
We therefore extend the formal immersion method by considering higher-order terms in the formal expansions at the cusps. 
This approach yields the following upper bound for the sum of the degrees of the finite places at which the Drinfeld module associated with an $\mathbf{F}_q(T)$-rational point of $Y_0(\mathfrak p)$ does not have potentially good reduction. 

\begin{theorem}[{Theorem \ref{thm:bad-reduction-bound}}]\label{intro:bad-reduction-bound}
For any rational point $x_0\in Y_0(\mathfrak p)(\mathbf{F}_q(T))$, one has 
\[
\sum_{\substack{\text{$v$ finite},\ v\ne\mathfrak p\\
                 v(j(x_0))<0}}
\deg(v)
\le \text{the genus of $X_0(\mathfrak p)$}.
\]
Here, $j(x_0)\in \mathbf{F}_q(T)$ is the $j$-invariant of the Drinfeld module corresponding to $x_0$.  
\end{theorem}

By combining Theorem~\ref{intro:bad-reduction-bound} with a refinement of Ishii's results on isogeny characters \cite{Ish22}, we establish our main results on rational points of Drinfeld modular curves.

\subsection*{Acknowledgements} 
The author is deeply grateful to Shun Ishii for explaining
the difficulties in applying the formal immersion method
to Drinfeld modular curves and how potentially good reduction
at finite places of small degree can be used to establish
the nonexistence of rational points on Drinfeld modular curves.
The idea of counting the finite places at which a Drinfeld module
does not have potentially good reduction was inspired
by his explanations. 
The author also wishes to express his appreciation to Sho Yoshikawa for giving the initial impetus for this research. 
The author was supported by JSPS KAKENHI Grant Number JP24K16886.

\subsection*{Declaration of AI Use}

During the preparation of this manuscript, the authors used ChatGPT to assist with language editing and to improve the clarity of the mathematical exposition. 
All AI-generated suggestions were reviewed and, where necessary, revised by the author. 
For the computer-assisted calculation in Theorem~\ref{thm:even}, the author used AI to assist in writing the code and carrying out the computation. 
The author takes full responsibility for the mathematical content of this manuscript. 

\subsection*{Notation} 
Let $q$ be a prime power. 
Put $A := \mathbf{F}_q[T]$ and $F := \mathbf{F}_q(T)$. 
In this paper, we assume that $q \geq 3$. 
For a closed point $v$ of $\mathbf P^1_{\mathbf F_q}$, let $F_v$ be the
completion of $F$, let $\mathcal O_v$ be its valuation ring, and let $k_v$
be its residue field. Valuations are normalized to have value group
$\mathbf Z$. We put $\deg v :=[k_v:\mathbf F_q]$ and normalize
$|T|_\infty=q$. Let $\mathbf C_\infty$ be the completion of an algebraic
closure of $F_\infty$.

Let $\mathfrak{p} = (P)$ be a nonzero prime ideal of $A$,
where $P(T)$ is monic and irreducible, and put $d := \deg(\mathfrak{p})$. 
Suppose that $d \geq 3$. 

\section{Drinfeld Modular Curves}\label{sec:eisenstein} 

\subsection{Preliminaries}\label{sec:prelim}
We first recall well-known facts about Drinfeld modular curves.
Let $\tau$ denote the $q$-power Frobenius operator.
A rank-two Drinfeld $A$-module over $F$ has the form
\[
\phi_T=T+g\tau+\Delta\tau^2,
\qquad g\in F,\qquad\Delta\in F^\times,
\qquad j(\phi) := g^{q+1}/\Delta.
\]
Let $Y_0(\mathfrak p)$ and $Y_1(\mathfrak p)$ denote the coarse moduli curves over $F$ parametrizing rank-two Drinfeld $A$-modules equipped with a cyclic $\mathfrak p$-subgroup and a point of exact order $\mathfrak p$, respectively. 

\begin{lemma}\label{lem:descent}
The following assertions hold.
\begin{enumerate}[label=(\arabic*)]
\item Every point of $Y_0(\mathfrak p)(F)$ is represented by a pair $(\phi,C)$,
where $\phi/F$ is a rank-two Drinfeld $A$-module and
$C\subset\phi[\mathfrak p]$ is a cyclic subgroup defined over $F$.
\item Every point of $Y_1(\mathfrak p)(F)$ is represented by a pair $(\phi,a)$,
where $\phi/F$ is a rank-two Drinfeld $A$-module and
$a\in\phi[\mathfrak p](F)\setminus\{0\}$.
\end{enumerate}
\end{lemma}
\begin{proof}
For assertion (1), see \cite[Lemma~5.4]{Ish22} for example.
\end{proof}

\begin{lemma}\label{lem:genus}
The smooth projective compactification $X_0(\mathfrak p)$ of $Y_0(\mathfrak{p})$ has two $F$-rational cusps and the Atkin--Lehner involution $W_{\mathfrak p}$
interchanges them.
Moreover, the genus of  $X_0(\mathfrak p)$  is given by 
\[
g_{\mathfrak p} := 
\begin{cases}
(q^d-q)/(q^2-1), & \text{$d$ is odd},
\\
(q^d-q^2)/(q^2-1), & \text{$d$ is even}.
\end{cases}
\]
\end{lemma}
\begin{proof}
See \cite[Theorem~3.1, Corollary~3.4, and p.~226]{Gek86}.
\end{proof}

Note that $g_{\mathfrak p} > 0$ since we assume that $d \geq 3$.  
To distinguish the infinite place $\infty$ of $F$ from the
cusps  of $X_0(\mathfrak p)$, we denote by
$c_0, c_{\infty}\colon\operatorname{Spec}(F)\to X_0(\mathfrak p)$
the sections corresponding to the cusps $0$ and $\infty$, respectively. 
We use the same notation $c_0$ and $c_\infty$ for their extensions to the integral models of $X_0(\mathfrak p)$ considered below.

\subsection{The canonical integral compactification}\label{sec:canonical-compact}

Choose a Carlitz period $\widetilde\pi$ and let $e_A$ be the exponential
with lattice $A$, normalized by $e_A'(0)=1$. Put
\[
t(z) := \frac{1}{\widetilde\pi e_A(z)},\qquad x := t^{q-1},\qquad y := Tx.
\]
The standard expansion $j(x) \in -x^{-1}+A[[x]]$ identifies
$\widehat{\mathcal O}_{X_0(\mathfrak{p}), c_{\infty}} = F[[x]]$; see \cite[\S2, p.~224]{Gek86}.

Let $\mathcal X_0(\mathfrak p)$ be the canonical integral compactification
over $A$.
The finite morphism $j\colon \mathcal X_0(\mathfrak p)\to\mathbf P^1_A$ identifies
$\mathcal X_0(\mathfrak p)$ with the normalization of $\mathbf P^1_A$
in $F(X_0(\mathfrak p))$. The cusp sections are disjoint and smooth, including
over $\mathfrak p$; see \cite[(5.2), (5.3), (5.7), and the proof of (5.11)]{Gek86}
and \cite[Theorem~4.2]{PW15}.
This normal integral model is flat over $\operatorname{Spec}(A)$.

Let $\mathcal{I}_\infty$ denote the ideal sheaf of the cusp section
$c_{\infty}$ in $\mathcal X_0(\mathfrak p)$.
We denote the formal completion of $\mathcal X_0(\mathfrak p)$
along $c_{\infty}$ by
\[
\widehat{\mathcal X}_{0}(\mathfrak p)_{c_{\infty}}
:=
\varprojlim_{n\ge 0}
\left(
\mathcal X_0(\mathfrak p),
\mathcal O_{\mathcal X_0(\mathfrak p)}/\mathcal{I}_{\infty}^{n+1}
\right).
\]

\begin{lemma}\label{lem:cusp-inputs}
The following statements hold. 
\begin{itemize}
    \item[(1)] For every finite place $v$, the parameter $x$ induces
an isomorphism $\widehat{\mathcal O}_{\mathcal X_0(\mathfrak p),c_{\infty}(v)}  \cong \mathcal O_v[[x]]$. 
    In particular, one has $\widehat{\mathcal X}_{0}(\mathfrak p)_{c_{\infty}}
\cong \operatorname{Spf}A[[x]]$  and 
\[
c_{\infty}^* \Omega^1_{\mathcal X_0(\mathfrak p)/\operatorname{Spec}(A)} \cong A \, \mathrm{d}x. 
\]
    \item[(2)] Let $Y$ be the coordinate on $\mathbf A^1_{F_\infty}$.
Define the Berkovich open unit disk over $F_\infty$ by
\[
\mathbb{D} := \{\xi \in \mathbf{A}^{1,\mathrm{an}}_{F_\infty} \mid |Y|_\xi<1\},
\]
with its induced analytic structure.
Then there is an $F_\infty$-analytic open immersion
\[
\iota_\infty \colon \mathbb D
\hookrightarrow X_0(\mathfrak{p})_{F_\infty}^{\mathrm{an}},
\qquad
\iota_\infty(0) = c_{\infty},
\]
whose induced map on completed local rings is
$\iota_\infty^* \colon  F_\infty[[x]] \stackrel{\sim}{\to} F_\infty[[Y]]; y = Tx\mapsto Y$.
\end{itemize}
\end{lemma}

\begin{proof}
Since the morphism $j$ is finite, the divisor $j^{-1}(\infty_A)$ does not have vertical components.
Moreover, $j$ has a simple pole at the cusp $c_{\infty}$ on the generic fiber.
Hence, near the section $c_{\infty} \in \mathcal{X}_0(\mathfrak{p})(A)$, the divisor $j^{-1}(\infty_A)$ consists only of that section with multiplicity one.
Consequently, one has
$\widehat{\mathcal O}_{\mathcal X_0(\mathfrak p),c_{\infty}(v)} = \mathcal O_v[[j^{-1}]] \cong \mathcal O_v[[x]]$. 
The equality follows from the fact that $j(x) \in -x^{-1}+A[[x]]$.

For assertion (2), we first construct a morphism $\mathbb D \to X_0(\mathfrak{p})_{F_\infty}^{\mathrm{an}}$.
Put $j_P(z):=j(Pz)$. 
The Carlitz division polynomials and the cusp expansions give
$j,j_P\in F((x))=F((y))$, convergent on the punctured disk,
and $F(X_0(\mathfrak p))=F(j,j_P)$ by
\cite[\S1--2 and Proposition~3.2(iii)]{Gek86}.
Thus $(j,j_P)$ defines an $F_\infty$-analytic morphism 
$\mathbb D \to \mathbf P^{1,\mathrm{an}}_{F_\infty} \times\mathbf P^{1,\mathrm{an}}_{F_\infty}$. 
Since $F_\infty(X_0(\mathfrak p))=F_\infty(j,j_P)$,
it follows from the facts that $\mathbb D$ is normal that the nonconstant morphism $(j,j_P) \colon \mathbb D \to \mathbf P^{1,\mathrm{an}}_{F_\infty} \times \mathbf P^{1,\mathrm{an}}_{F_\infty}$ factors uniquely through a morphism
$\iota_\infty\colon\mathbb D \to  X_0(\mathfrak p)_{F_\infty}^{\mathrm{an}}$. 

By \cite[Theorem~11.5]{CT21}, it remains to show that the base change of $\iota_\infty$ to $\mathbf C_\infty$ is an open immersion. 
Let $|\cdot|$ denote the absolute value on $\mathbf{C}_\infty$ defined by $|f| := q^{\deg f}$ for any nonzero $f \in A$. 
Put $\Omega:=\mathbf C_\infty\setminus F_\infty$ and
$|z|_i:=\inf_{a\in F_\infty}|z-a|$.
Let $\Gamma_\infty$ be the stabilizer of $\infty$ in
\[
\Gamma_0(\mathfrak p):=
\left\{
\begin{pmatrix}
\alpha&\beta\\
\gamma&\delta
\end{pmatrix}
\in\mathrm{GL}_2(A)
\,\middle|\,
\gamma\in\mathfrak p
\right\}.
\]
By \cite[Lemma~23]{Tei91}, using
$|\widetilde\pi|=q^{q/(q-1)}$ and then quotienting by
$\mathbf F_q^\times$, one obtains
\[
\{z\in\Omega\mid |z|_i>q^{-1}\}/\Gamma_\infty \stackrel{\sim}{\to}
\{0<|y|<1\}; \qquad z \mapsto y(z). 
\]
If
$g=\begin{pmatrix}\alpha&\beta\\
\gamma&\delta\end{pmatrix}\in\Gamma_0(\mathfrak p)$
has $\gamma\ne0$, then $|\gamma|\ge q^d$ and
\[
|gz|_i
=\frac{|z|_i}{|\gamma z+\delta|^2}
\le\frac{1}{|\gamma|^2|z|_i}<q^{-1}.
\]
Hence
$\{z\in\Omega\mid |z|_i>q^{-1}\}/\Gamma_\infty$
embeds in
$\Gamma_0(\mathfrak p)\backslash\Omega
\cong Y_0(\mathfrak p)_{\mathbf C_\infty}^{\mathrm{an}}$.
The embedding extends canonically to an open immersion $\iota'_{\infty} \colon \mathbb D_{\mathbf C_\infty} \hookrightarrow
X_0(\mathfrak p)_{\mathbf C_\infty}^{\mathrm{an}}$. 
By construction, one has $(j,j_P)\circ(\iota_\infty)_{\mathbf C_\infty} = (j,j_P)\circ\iota'_\infty$ on the punctured disk, and hence on $\mathbb D_{\mathbf C_\infty}$. 
Thus $(\iota_\infty)_{\mathbf C_\infty} =\iota'_\infty$ by the uniqueness of $\iota_\infty$. 
\end{proof}

\subsection{The Eisenstein quotient}

Put $J := \operatorname{Jac}(X_{0}(\mathfrak p))$, and let $B := E(\mathfrak p)$ be the Eisenstein
quotient defined in \cite[Definition~3.2]{Pal10}.
By definition, the kernel of the natural quotient map
$\pi \colon J\twoheadrightarrow B$ is an abelian subvariety of $J$.
In particular, $\pi$ is smooth.
We define
\[
u \colon X_0(\mathfrak p)\to B, \qquad u(x) := \pi([x]-[c_{\infty}]).
\]

\begin{theorem}\label{thm:MW}
The quotient map induces an isomorphism
\[
\langle[c_{0}]-[c_{\infty}]\rangle  \stackrel{\sim}{\to} B(F),
\qquad
\#B(F)=
\begin{cases}
(q^d-1)/(q-1),& \text{$d$ is odd},\\
(q^d-1)/(q^2-1),& \text{$d$ is even}.
\end{cases}
\]
In particular, $\dim B>0$, $B(F)$ is finite, and $B(F)[q^\infty]=0$.
Moreover, $\pi\circ W_{\mathfrak p}=-\pi$, and $B$ has split purely toric reduction at $\infty$.
The homomorphism
\[
u^* \colon H^0(B,\Omega^1_{B/F}) \hookrightarrow
H^0\bigl(X_0(\mathfrak p),\Omega^1_{X_0(\mathfrak p)/F}\bigr)
\]
is injective. In particular, $u$ is nonconstant.
\end{theorem}
\begin{proof}
The isomorphism and the Atkin--Lehner action are proved by P\'al in
\cite[Theorem~3.5 and Corollary~3.7]{Pal10}.
The order of the cuspidal divisor is proved by Gekeler in \cite[Theorem~4.1]{Gek86}.
The Jacobian $J$ has split purely toric reduction at $\infty$ by \cite[\S3.1]{Pap04}.
The same holds for $B$ by \cite[Lemma~2.2]{Pap04}.
Since $H^0(J,\Omega^1_{J/F}) \stackrel{\sim}{\to}H^0(X_0(\mathfrak p),
\Omega^1_{X_0(\mathfrak p)/F})$, the injectivity of $u^*$ follows from the fact that $\pi$ is smooth and surjective, and $\dim B>0$ implies that $u$ is nonconstant.
\end{proof}

\subsection{Nefness of the Hodge bundle}\label{sec:hodge}
Let $\mathcal B$ be the N\'eron model of $B$ over
$\mathbf P^1_{\mathbf F_q}$, and let $\mathcal B^0$ be its identity
component. 
Let $e_{\mathcal{B}} \colon \mathbf P^1_{\mathbf{F}_q}\to\mathcal{B}^0$ be the identity section. We define the Hodge bundle by
\[
\mathcal{E}_{\mathcal{B}} := e_{\mathcal{B}}^*\Omega^1_{\mathcal{B}^0/\mathbf P^1_{\mathbf{F}_q}}.
\]

\begin{lemma}\label{lem:nef}
The abelian variety $B$  is ordinary, that is, $\dim_{\mathbf F_p}B[p](\overline F)= \dim B$, and the bundle $\mathcal{E}_{\mathcal{B}}$
on $\mathbf P^1_{\mathbf F_q}$ is nef.
\end{lemma}
\begin{proof}
Since $B$ has split purely toric reduction at $\infty$ by Theorem \ref{thm:MW}, Raynaud's uniformization gives $B^{\mathrm{an}}=G^{\mathrm{an}}/\Lambda$ over $\mathbf{C}_\infty$,
where $G\cong\mathbf G_m^{\dim B}$ and
$\Lambda\cong\mathbf Z^{\dim B}$.
Multiplication by $p$ is bijective on $G(\mathbf{C}_\infty)$.
Consequently $B[p](\mathbf C_\infty)\cong\Lambda/p\Lambda$ and the $p$-rank of $B$ is $\dim B$.

Write $\mathcal{E}_{\mathcal{B}}\cong\bigoplus_i\mathcal O_{\mathbf P^1_{\mathbf F_q}}(d_i)$.
Since the pullback of $\mathcal O_{\mathbf P^1_{\mathbf F_q}}(d')$ by absolute $p$-Frobenius is $\mathcal{O}_{\mathbf P^1_{\mathbf F_q}}(p d')$, we have
\[
\lim_{n\to\infty}p^{-n}
\mu_{\min}\bigl((\operatorname{Frob}_p^n)^*\mathcal{E}_{\mathcal{B}}\bigr)
=\min_i d_i.
\]
Here $\mu_{\min}$ denotes the minimum slope.
Since ordinarity is preserved under field extensions, R\"ossler's theorem \cite[Theorem~1.2(a)]{Ros15} applies after base change to $\mathbf P^1_{\overline{\mathbf F}_q}$, yielding $\min_i d_i\ge0$.
Thus $\mathcal{E}_{\mathcal{B}}$ is nef.
\end{proof}

\section{An upper bound for primes with no potentially good reduction}\label{sec:bound}

\subsection{Extension at the infinite place}

\begin{lemma}\label{lem:infinity}

Let $E/F_\infty$ be an abelian variety with split purely toric reduction, and let $\mathcal E/\mathcal{O}_\infty$ be its N\'eron model.
Denote by $\widehat{\mathcal E}_0$ the formal
completion of $\mathcal E$ at the identity section of its special fiber,
with the maximal-ideal-adic topology.
By Lemma~\ref{lem:cusp-inputs}(2), we identify $\mathbb{D}$
with an admissible open subset of
$X_0(\mathfrak p)_{F_\infty}^{\mathrm{an}}$ containing $c_{\infty}$,
with $c_{\infty}$ corresponding to $y=0$.

Then for every morphism $f \colon X_0(\mathfrak p)_{F_\infty}\to E$ with $f(c_{\infty})=0$, there exists a unique morphism of formal
$\mathcal{O}_\infty$-schemes
\[
\widehat f\colon
\operatorname{Spf}\mathcal O_\infty[[y]]
\to
\widehat{\mathcal E}_0,
\]
whose analytic generic fiber agrees with $f^{\mathrm{an}}|_{\mathbb D}$.
\end{lemma}

\begin{proof}
Put $g :=\dim E$.
Since $E$ has split purely toric reduction at $\infty$, Raynaud's uniformization gives an analytic covering
\[
\rho\colon G^{\mathrm{an}}\to E^{\mathrm{an}},
\qquad
G := \mathbf{G}_{m,F_\infty}^{g},
\]
whose kernel is a discrete  lattice. 
Let $\mathcal{E}_0$ denote the open subgroup scheme of $\mathcal{E}$ whose fibers are the identity components of the fibers of $\mathcal{E}$.
Let $\widehat{\mathcal E}_{0, k_{\infty}}$ denote the formal completion of $\mathcal{E}_0$ along its special fiber.
Then one has $\operatorname{Spf}\mathcal O_\infty
 \langle Z_1^{\pm1},\ldots,Z_g^{\pm1}\rangle
\cong \widehat{\mathcal{E}}_{0, k_{\infty}}$ as $T^{-1}$-adic formal schemes 
(see \cite[\S3.1]{Pap04} and \cite[the proof of Proposition~3.3]{Pap07}). 
Completing further at the identity point $Z_i=1$ of the special fiber gives
\[
\widehat{\mathcal E}_0
\cong\operatorname{Spf}\mathcal O_\infty[[X_1,\ldots,X_g]],
\qquad X_i := Z_i-1,
\]
with  $(T^{-1},X_1,\ldots,X_g)$-adic topology. The analytic
generic fibers of these completed local formal schemes are
the open unit polydisks.

A Berkovich open disk is a locally path-connected tree; see \cite[Propositions~5.7.1 and~5.7.12,
Remark~5.7.14]{PT20I} for example.  
Hence the morphism $f$ admits a unique analytic lift
\[
\widetilde{f} \colon\mathbb{D}
\to G^{\mathrm{an}},
\qquad
\widetilde f(0)=(1,\ldots,1).
\]
Write
\[
\widetilde f(y)=(h_1(y),\ldots,h_g(y)),
\qquad
h_i(y)=1+\sum_{n\ge1}c_{i,n}y^n,
\qquad c_{i,n}\in F_\infty.
\]
Both $h_i$ and $1/h_i$ are analytic on
$|y|<1$.
For $0<r<1$, let $\|\cdot\|_r$ denote the
Gauss norm of radius $r$.
By multiplicativity of the Gauss norm and $h_i(0)=1$, one has
\[
\|h_i\|_r\|h_i^{-1}\|_r=1,
\qquad
\|h_i\|_r\ge1,\qquad
\|h_i^{-1}\|_r\ge1.
\]
Consequently, $\|h_i\|_r=1$ and $|c_{i,n}|r^n\le1$.
Letting $r\to1$ gives $c_{i,n}\in\mathcal O_\infty$ and $h_i(y)-1\in y\mathcal O_\infty[[y]]$.
The homomorphism
\[
\mathcal O_\infty[[X_1,\ldots,X_g]]
\to
\mathcal O_\infty[[y]],
\qquad
X_i\mapsto h_i(y)-1,
\]
is continuous for these maximal-ideal-adic topologies and therefore
defines the required morphism
\[
\widehat{f} \colon  \operatorname{Spf}\mathcal O_\infty[[y]]
\to
\widehat{\mathcal E}_0.
\]
On $|y|<1$, the inequalities $|h_i(y)-1|\le |y|<1$ show that
its analytic generic fiber is precisely $f^{\mathrm{an}}$.
Uniqueness follows from the uniqueness of
the normalized analytic lift and the injectivity of $\mathcal O_\infty[[y]] \hookrightarrow F_\infty[[y]]$.
\end{proof}

\subsection{The first nonzero term at the cusp \texorpdfstring{$c_{\infty}$}{c\_infty}}

Let $\mathfrak{m}_{B,0}$ and $\mathfrak{m}_{X_{0}(\mathfrak p), c_{\infty}}$ denote the maximal ideals at the identity $e_{B}\in B$ and the cusp $c_{\infty} \in X_{0}(\mathfrak{p})$, respectively. 
Since $u$ is nonconstant and $\mathcal O_{X_0(\mathfrak{p}), c_{\infty}}$ is a discrete
valuation ring, there is a unique integer $m\ge1$ such that
\begin{equation}\label{eq:m}
u^*\mathfrak m_{B,0}\cdot\mathcal O_{X_0(\mathfrak{p}), c_{\infty}}
=\mathfrak m_{X_0(\mathfrak{p}), c_{\infty}}^{m}.
\end{equation}
Equivalently, $m$ is the least degree of a nonzero coefficient among
the expansions $u^*z_i\in F[[x]]$, for local parameters $z_1, \ldots, z_{\dim B}$ of $B$ at $0$.
Since $2m\ge m+1$, pullback induces a nonzero $F$-linear map
\[
\alpha_m \colon 
\mathcal{E}_{\mathcal{B}}\otimes F
\cong\mathfrak m_{B,0}/\mathfrak m_{B,0}^2
\to
\mathfrak m_{X_0(\mathfrak p), c_{\infty}}^{m}/
\mathfrak m_{X_0(\mathfrak p), c_{\infty}}^{m+1} \cong\operatorname{Cot}_{ c_{\infty}}(X_0(\mathfrak p))^{\otimes m}.
\]
If $u^*z_i = a_{i,m}x^m+O(x^{m+1})$, then $\alpha_m(\mathrm{d}z_i) = a_{i,m} \, (\mathrm{d}x)^{\otimes m}$.

The finite-place parameter $x$ and the infinite-place
parameter $y=Tx$ (see Section \ref{sec:canonical-compact}) determine the following
$\mathcal O_v$-lattices in
$\operatorname{Cot}_{ c_{\infty}}\bigl(X_0(\mathfrak p)\bigr)$:
\[
\mathcal L_v := \mathcal O_v\,\mathrm{d}x
\quad(v\ne\infty),\qquad
\mathcal L_\infty := \mathcal O_\infty\,\mathrm{d}y
=T\mathcal O_\infty\,\mathrm{d}x.
\]
Since $\operatorname{ord}_\infty(T)=-1$,  these lattices define a line bundle $\mathcal L :=\mathcal O_{\mathbf P^1_{\mathbf F_q}}(\infty)\,\mathrm{d}x$
on $\mathbf{P}^1_{\mathbf{F}_q}$. 
Note that $\mathcal{L} \otimes F = \operatorname{Cot}_{c_{\infty}}(X_0(\mathfrak p))$.
Moreover, by Lemma \ref{lem:cusp-inputs}(1) one has
\begin{align}\label{eq:finite}
c_{\infty}^*\Omega^1_{\mathcal{X}_0(\mathfrak{p})/\operatorname{Spec}(A)} \cong \mathcal{L}|_{\operatorname{Spec}(A)}. 
\end{align}

\begin{theorem}\label{thm:jetdivisor}
The map $\alpha_m \colon  \mathcal{E}_{\mathcal{B}}\otimes F
\to
\mathcal L^{\otimes m}\otimes F$  extends uniquely to a nonzero morphism
$\alpha_m \colon \mathcal{E}_{\mathcal{B}}\to\mathcal{L}^{\otimes m}$ of vector bundles on $\mathbf P^1_{\mathbf F_q}$.
Moreover, there is a unique effective divisor $D_m$ on $\mathbf P^1_{\mathbf F_q}$ with
$\operatorname{im}\alpha_m=\mathcal L^{\otimes m}(-D_m)$, and  $\deg D_m\le m\le g_{\mathfrak p}$.
\end{theorem}

\begin{proof}
Let $\mathcal I$ denote the ideal sheaf of the cusp section
$c_{\infty} \colon \operatorname{Spec}(A) \to \mathcal X_0(\mathfrak p)$.
By the N\'eron mapping property, $u$ extends to a morphism $u_A \colon \mathcal X_0(\mathfrak p)^{\mathrm{sm}} \to \mathcal{B}_A := \mathcal{B}|_{\operatorname{Spec}(A)}$. 
Then $u_A$ maps $c_{\infty}$ to the identity section of $\mathcal{B}_A$. 
Let $\mathcal J$ denote the ideal sheaf of the identity section
of $\mathcal{B}_A$.
Since $\widehat{\mathcal O}_{X_0(\mathfrak{p}), c_{\infty}} = F[[x]]$, the definition of $m$ shows that  $u_A^*\mathcal J\otimes_A F \subseteq \mathcal I^m\otimes_A F$. 
Since $\mathcal O_{\mathcal X_0(\mathfrak p)}/\mathcal I^m$
is flat over $A$, it follows that  $u_A^*\mathcal J\subseteq\mathcal I^m$. 
This gives an $A$-linear map
\[
\mathcal{E}_{\mathcal{B}}|_{\operatorname{Spec}(A)}
\to
\mathcal I^m/\mathcal I^{m+1}
\cong 
\mathcal L^{\otimes m}|_{\operatorname{Spec}(A)},
\]
where the isomorphism follows from \eqref{eq:finite}.
The induced map agrees with $\alpha_m$ on the generic fiber.
Thus $\alpha_m$ extends over $\operatorname{Spec}(A)$.
Moreover, $\alpha_m(\mathcal{E}_{\mathcal{B},\infty})
\subseteq\mathcal L_{\infty}^{\otimes m}$ follows from Lemma~\ref{lem:infinity}, applied to $u$ with parameter $y$.
Since $\mathbf P^1_{\mathbf F_q}$ is a smooth curve, these facts give a unique nonzero morphism $\alpha_m \colon \mathcal{E}_{\mathcal{B}}\to
\mathcal L^{\otimes m}$.

The nonzero subsheaf $\operatorname{im}\alpha_m\subseteq\mathcal L^{\otimes m}$ on $\mathbf P^1_{\mathbf F_q}$ is an invertible sheaf. Hence there exists a unique
effective divisor $D_m$ on $\mathbf P^1_{\mathbf F_q}$ such that $\operatorname{im}\alpha_m = \mathcal L^{\otimes m}(-D_m)$.
Since $\operatorname{im}\alpha_m$ is a quotient line
bundle of the nef bundle $\mathcal{E}_{\mathcal{B}}$ by Lemma \ref{lem:nef}, we have $0\le\deg(\operatorname{im}\alpha_m)
=m-\deg D_m$.

By Theorem~\ref{thm:MW}, one has 
\[
u(W_{\mathfrak p}x)= \pi([W_{\mathfrak p}x] - [c_{\infty}]) = - \pi([x] - [c_{0}]) = -u(x)+u(c_{0}), 
\]
and hence $W_{\mathfrak p}^*(u^*\omega)=-u^*\omega$ for every $\omega\in H^0(B,\Omega^1_{B/F})$. 
Every local coordinate of $u$ at $c_{\infty}$ has order at least $m$,
so its differential has order at least $m-1$. Hence
$\operatorname{ord}_{c_{0}}(u^*\omega)=\operatorname{ord}_{c_{\infty}}(u^*\omega)\ge m-1$ whenever $u^*\omega\ne0$.
By Theorem~\ref{thm:MW}, there exists an invariant
differential $\omega$ such that $u^*\omega$ is nonzero.
We conclude that $2g_{\mathfrak p}-2 = \deg\operatorname{div}(u^*\omega) \geq \operatorname{ord}_{c_{0}}(u^*\omega) + \operatorname{ord}_{c_{\infty}}(u^*\omega) = 2(m-1)$.
\end{proof}

\begin{remark}
With some additional work, one can construct a ``good''
% proper, flat, normal
model $\mathscr{X}_0(\mathfrak p)$ of $X_0(\mathfrak p)$ over
$\mathbf P^1_{\mathbf F_q}$ with a section
$\tilde{c}_{\infty} \colon \mathbf P^1_{\mathbf F_q}
\to\mathscr{X}_0(\mathfrak p)$ 
such that $\mathscr{X}_0(\mathfrak p)|_{\operatorname{Spec}(A)}
\cong \mathcal X_0(\mathfrak p)$,  $\tilde{c}_{\infty}|_{\operatorname{Spec}(A)} = c_{\infty}$, and 
\[
\tilde{c}_{\infty}^*
\Omega^1_{\mathscr{X}_0(\mathfrak p)/\mathbf P^1_{\mathbf F_q}}
\cong \mathcal L.
\]
Such a model would provide a more natural framework for the
construction of $\mathcal L$ and the proof of
Theorem~\ref{thm:jetdivisor}.
However, we use a direct, slightly ad hoc argument to avoid the additional technical details involved in constructing $\mathscr{X}_0(\mathfrak p)$. 
\end{remark}

\subsection{Primes at which Drinfeld modules do not have potentially good reduction}\label{sec:bad-reduction}

A rank-two Drinfeld module has good reduction at a finite place $v$
if, after an isomorphism over $F_v$, its coefficients are integral
and its leading coefficient is a unit. It has potentially good
reduction at $v$ if this holds after a finite extension of $F_v$.
For $\phi_T=T+g\tau+\Delta\tau^2$,  it is
well-known that 
$\phi$ has potentially good reduction at $v$ if and only if
$v(g^{q+1}/\Delta)\ge0$; see \cite[\S2 and Lemma~2.3]{Ish22}.
Potentially good reduction is invariant under isogeny and under
twist.

\begin{theorem}\label{thm:bad-reduction-bound}
If $x_0\in Y_0(\mathfrak p)(F)$, then one has
\[
 \sum_{\substack{\text{$v$  finite}, \ v\ne\mathfrak p\\
                  v(j(x_0))<0}}\deg(v) \le g_{\mathfrak p}.
\]
\end{theorem}
\begin{proof}
Let $v\ne\mathfrak p$ be a finite place with $v(j(x_0))<0$.
By the properness of $\mathcal X_0(\mathfrak p)$, the point $x_0\in X_0(\mathfrak p)(F)$ extends uniquely to an $\mathcal O_v$-point $\widetilde{x}_0\in\mathcal X_0(\mathfrak p)(\mathcal O_v)$. 
Since $j^{-1}(\infty_{\mathcal{O}_v})\cap \mathcal X_0(\mathfrak p)_{k_v} = \{c_{0}(v),c_{\infty}(v)\}$, where $j \colon \mathcal{X}_0(\mathfrak p)_{\mathcal O_v}
\to \mathbf{P}^1_{\mathcal O_v}$, one has $\tilde{x}_0(v) \in \{c_{0}(v),c_{\infty}(v)\}$. 
The involution $W_{\mathfrak p}$ extends over $A[1/P]$ and interchanges the cusp sections $c_{0}$ and $c_{\infty}$; see \cite[\S5]{Gek86}.
Replacing $x_0$ by $W_{\mathfrak p}x_0$  if necessary,
we may assume that $\tilde{x}_0(v) = c_{\infty}(v)$. 
Then the point $u(x_0)$ belongs both to $B(F)$ and
to $\ker(\mathcal B(\mathcal O_v)\to\mathcal B(k_v))$.
The former has finite order prime to $q$ by Theorem~\ref{thm:MW}.
The latter has no prime-to-$q$ torsion, since multiplication by an
integer prime to $q$ is an automorphism of the smooth formal group.
Hence $u(x_0) = e_{B}$.

Suppose that $v\notin\operatorname{Supp}D_m$.
By Theorem~\ref{thm:jetdivisor}, the map $\alpha_{m,v}\colon
\mathcal E_{B,v}\to
\mathcal O_v \, (\mathrm{d}x)^{\otimes m}$ is surjective.
Let $u_v \colon \mathcal X_0(\mathfrak p)_{\mathcal{O}_v} \to  \mathcal{B}_{\mathcal{O}_v}$ denote the unique extension of $u$. 
Choose formal coordinates $Z_1,\ldots,Z_g$
vanishing along the identity section $e_{\mathcal{B}_{\mathcal{O}_v}}$ of $\mathcal{B}_{\mathcal{O}_v}$, where $g=\dim B$. 
Recall $\widehat{\mathcal O}_{\mathcal X_0(\mathfrak p),c_{\infty}(v)}  \cong \mathcal O_v[[x]]$ by Lemma \ref{lem:cusp-inputs}(1). 
By the definition of $m$, one has 
\[
u_v^*Z_i=\sum_{n\ge m}a_{i,n}x^n, \qquad 
a_{i,n}\in\mathcal O_v, \qquad 
\alpha_{m,v} \, (\mathrm{d}Z_i) = a_{i,m} \, (\mathrm{d}x)^{\otimes m}.
\]
It follows from the surjectivity of $\alpha_{m,v}$ that there are  $b_1,\ldots,b_g\in\mathcal O_v$
such that $\sum_{i=1}^g b_i a_{i,m}=1$.
Consequently, for $z=\sum_{i=1}^g b_iZ_i$, one has
\[
u_v^*z=x^m(1+xh(x)),
\qquad h(x)\in\mathcal O_v[[x]].
\]
Since $\tilde{x}_0(v) = c_{\infty}(v)$, 
the section $\tilde{x}_0$ defines $\tilde{x}_0^{*} \colon \mathcal{O}_v[[x]] \to \mathcal{O}_v$. 
Hence one has $\tilde{x}_0^{*}(x) \in \mathfrak m_v$, where $\mathfrak m_v \subset \mathcal O_v$ is the maximal ideal.
Thus $1+\tilde{x}_0^{*}(x)h(\tilde{x}_0^{*}(x)) \in\mathcal O_v^\times$. 
Since $u(x_0)=e_{B}$, one has $u_v(\tilde{x}_0) = e_{\mathcal{B}_{\mathcal{O}_v}}$. 
This gives
\[
0 =  e_{\mathcal{B}_{\mathcal{O}_v}}^*(z) = \tilde{x}_0^{*}(u^*_vz) = \tilde{x}_0^{*}(x)^m  \bigl(1+\tilde{x}_0^{*}(x)h(\tilde{x}_0^{*}(x))\bigr).
\]
Hence $\tilde{x}_0^{*}(x) =0$, so $\tilde{x}_0^* = c_{\infty}^*$, contrary to $x_0\in Y_0(\mathfrak p)(F)$.
This proves $v\in\operatorname{Supp}D_m$.
Since $D_m$ is effective, it follows from Theorem \ref{thm:jetdivisor} that  \[
\sum_{\substack{v\text{ finite},\ v\ne\mathfrak p\\v(j(x_0))<0}}
\deg v
\le
\sum_{v\in\operatorname{Supp}D_m}\deg v
\le\deg D_m \le g_{\mathfrak p}.
\]
\end{proof}

\section{Isogeny characters}\label{sec:characters}

Let $\phi/F$ be a rank-two Drinfeld module with an $F$-rational cyclic $\mathfrak p$-subgroup $C$.
Put $G_F := \operatorname{Gal}(F^{\mathrm{sep}}/F)$.
The action on $C$ defines an isogeny character
$\lambda \colon G_F\to(A/\mathfrak p)^\times$.
Let $\kappa \colon G_F\to(A/\mathfrak p)^\times$ be the Carlitz character,
normalized by
\[
\kappa(\operatorname{Frob}_{(H)}) = H \bmod{\mathfrak{p}}
\]
for any monic irreducible polynomial $H \not\in \mathfrak{p}$.
Here $\operatorname{Frob}_{(H)}$ denotes the arithmetic Frobenius at $(H)$.
We choose decomposition and inertia groups $G_{F_v}$ and $I_v$ inside $G_F$.

In this section, we investigate the properties of the isogeny
character $\lambda$ that will be needed in the sequel.
We note that  the results presented in this section
are based on the work of Ishii \cite{Ish22}.
We recall the relevant results and arguments from his work
and formulate them in a form suitable for our purposes. 

\begin{lemma}\label{lem:Ishii}
There are $k\in\mathbf Z/(q^d-1)\mathbf Z$ and a character $\alpha \colon G_F\to(A/\mathfrak p)^\times$ such that
\[
\lambda=\alpha\kappa^k,
\qquad \alpha|_{I_{\mathfrak p}}=1,
\qquad \alpha^{\nu_d}=1,
\qquad
\nu_d := \begin{cases}q-1,&\text{if $d$ is odd},\\q^2-1,& \text{if $d$ is even}.
\end{cases}
\]
If $\phi$ has potentially good reduction at $\mathfrak p$, then
\[
\begin{cases}
(q+1)k\equiv n\pmod{(q^d-1)/(q-1)},\quad 0\le n\le q+1,&\text{if $d$ is odd},\\
k\equiv0,1\pmod{(q^d-1)/(q^2-1)},&\text{if $d$ is even}.
\end{cases}
\]
\end{lemma}
\begin{proof}
See \cite[Lemma~3.4, Proposition~3.5, and Lemma~3.6]{Ish22}.
\end{proof}

In what follows, let $k$ be the integer as in Lemma~\ref{lem:Ishii}.

\begin{lemma}\label{lem:bad-reduction-character}
If $v(j(\phi))<0$ at a finite place $v$, then
\[
\lambda^{q-1}|_{G_{F_v}}\in\{1,\kappa^{q-1}|_{G_{F_v}}\}.
\]
\end{lemma}
\begin{proof}
Replacing $\phi$ by an $\mathbf F_q^\times$-twist
changes the isogeny character $\lambda$ by multiplication
by a character with values in $\mathbf F_q^\times$.
Hence it does not change $\lambda^{q-1}$.
Such a twist puts $\phi$ in the form
$\phi_T=T+\tau+j(\phi)^{-1}\tau^2$, with stable reduction of rank one.
By Drinfeld's Tate uniformization \cite[Proposition~7.2]{Dri74}, we obtain the following Galois-equivariant short exact sequence 
\[
0\to\psi[\mathfrak p](F_v^{\mathrm{sep}})
\to\phi[\mathfrak p](F_v^{\mathrm{sep}})
\to\Lambda/\mathfrak p\Lambda\to0,
\]
where $\psi$ is a rank-one Drinfeld module with good reduction and $\Lambda$ is a Galois-stable rank-one $A$-lattice. 
The character on $\psi[\mathfrak p]$ is the product of
$\kappa$ with a character valued in $\mathbf F_q^\times$, whereas the action on
$\Lambda$ takes values in $A^\times=\mathbf F_q^\times$.
The subgroup $C$ either equals $\psi[\mathfrak p]$ or maps isomorphically
onto $\Lambda/\mathfrak p\Lambda$. Taking $(q-1)$-st powers proves the claim.
\end{proof}

\begin{lemma}\label{lem:power}
The following assertions hold.
\begin{enumerate}[label=(\arabic*)]
\item If $d$ is even, then $\lambda^{q^2-1}\in\{1,\kappa^{q^2-1}\}$.
\item If $d$ is odd and $v_{\mathfrak p}(j(\phi))<0$, then
$\lambda^{q-1}\in\{1,\kappa^{q-1}\}$.
\item If $d>q$ is odd and $\phi$ does not have potentially good reduction at
some finite place of degree one, then
$\lambda^{q-1}\in\{1,\kappa^{q-1}\}$.
\end{enumerate}
\end{lemma}
\begin{proof}
If $\phi$ has potentially good reduction at $\mathfrak p$, (1) follows
from Lemma~\ref{lem:Ishii}.
If $v_{\mathfrak p}(j(\phi))<0$, since $\kappa(I_{\mathfrak p})=(A/\mathfrak p)^\times$,
Lemmas~\ref{lem:Ishii} and \ref{lem:bad-reduction-character} give
\[
k(q-1)\equiv0,q-1\pmod{q^d-1}.
\]
This proves (1) and (2).
For (3), we may therefore assume potentially good reduction at $\mathfrak p$.
Let $H$ be a degree-one prime with $v_{(H)}(j(\phi)) < 0$. 
Put $h := \kappa(\operatorname{Frob}_{(H)}) \in A/\mathfrak{p}$.
Then it follows from Lemma \ref{lem:bad-reduction-character} that $h^{(q-1)k}\in\{1,h^{q-1}\}$.
Hence $h^{(q^2-1)k}\in\{1,h^{q^2-1}\}$, and Lemma \ref{lem:Ishii} implies $h^{(q-1)n}\in\{1,h^{q^2-1}\}$ for some $0 \leq n \leq q+1$.
In particular,  either $h^n$ or $h^{q+1-n}$ belongs to
$\mathbf F_q^\times$.
If $1\le n\le q$, this contradicts $[\mathbf F_q(h):\mathbf F_q]=d>q$.
Hence $n\in\{0,q+1\}$.
Since $\gcd(q+1,(q^d-1)/(q-1))=1$,
we obtain $k\equiv0,1\pmod{(q^d-1)/(q-1)}$ by Lemma \ref{lem:Ishii}.
\end{proof}

\begin{lemma}\label{lem:linear-good-reduction}
Suppose that $\phi$ has potentially good reduction at a finite
prime of degree one. Then $\phi$ has potentially good reduction at
$\mathfrak p$, and $d$ is odd.
If also $d>q$, then $q$ is odd and
\[
2k\equiv1\pmod{(q^d-1)/(q-1)}.
\]
\end{lemma}
\begin{proof}
If $v_{\mathfrak p}(j(\phi))<0$, the proof of Lemma~\ref{lem:power}
gives $k\equiv0,1\pmod{(q^d-1)/(q-1)}$.
Let $(H)$ be a  degree-one prime at which $\phi$ has potentially good reduction.
The argument preceding \cite[Theorem~4.2]{Ish22} gives $A_H\in\mathbf F_q$ such that
\[
A_H\equiv a_0H^k+b_0H^{1-k} \pmod{\mathfrak{p}}
\]
with $a_0,b_0\in\mathbf F_q^\times$ for odd $d$, and
$a_0,b_0\in\mathbf F_{q^2}^\times\subset(A/\mathfrak{p})^\times$ for even $d$.
Since $H^{(q^d-1)/(q-1)}\bmod{\mathfrak{p}}\in\mathbf F_q^\times$,
this forces $H\bmod{\mathfrak{p}}\in\mathbf F_{q^2}$, contradicting $d \geq 3$.
Thus $v_{\mathfrak p}(j(\phi)) \geq 0$.

Then it follows from \cite[Theorem~4.2]{Ish22}  that $d$ is odd.
If also $d>q$, \cite[Theorem~4.5]{Ish22}  excludes even $q$.
For odd $q$, \cite[Proposition~4.3]{Ish22} then gives
$2k\equiv1\pmod{(q^d-1)/(q-1)}$.
\end{proof}

\begin{proposition}\label{prop:supersingular}
Suppose that $q \equiv d \equiv 1 \pmod{2}$ and
$2k\equiv1\pmod{(q^d-1)/(q-1)}$.
Let $\phi'$ be the $\mathbf F_q^\times$-twist of $\phi$ whose isogeny character is $\kappa^k$.
Then $\phi'$ has good supersingular reduction at every  prime $(H)$
with $1\le\deg H<d$. Its Frobenius polynomial at $(H)$ is
\[
Z^2-\operatorname{Norm}_{(A/\mathfrak{p})/\mathbf F_q}(H)^rH,
\]
where $r := (2k-1)(q-1)/(q^{d}-1)$. 
\end{proposition}
\begin{proof}
We first note that the twist $\phi'$ exists because $\alpha^{q-1}=1$. 
Let $H \in A$ be a monic irreducible polynomial with $1\le\deg H<d$.
Put $h := H \bmod{\mathfrak{p}}$.
If $v_{(H)}(j(\phi))<0$, Lemma~\ref{lem:bad-reduction-character} gives $h^{(q-1)k}\in\{1,h^{q-1}\}$.
Using $2k(q-1)\equiv q-1\pmod{q^d-1}$ yields $h^{q-1} = 1$, that is,  $h\in\mathbf F_q^\times$.
This contradicts the fact that $1 \leq \deg H<d$.
Hence $v_{(H)}(j(\phi)) \geq 0$.

Since the isogeny character $\kappa^k$ of $\phi'$ is unramified at $(H)$, it follows from
\cite[Lemma~4.1]{Ish22} that  $\phi'$ has good reduction at $(H)$. 
By \cite[Theorem 1.8]{Gek08}, one can write the Frobenius polynomial as $Z^2-A_HZ+\delta_HH$, where
$\delta_H\in\mathbf F_q^\times$ and $2\deg A_H\le\deg H$.  
Then $z = h^k$ is a root of $Z^2-A_HZ+\delta_HH \equiv 0 \pmod{\mathfrak{p}}$. 
Moreover, it satisfies
\[
z^2=c_H h,\qquad
c_H := \operatorname{Norm}_{(A/\mathfrak{p})/\mathbf F_q}(h^r)\in\mathbf F_q^\times
\]
by the definition of $r$.
Using the other root $\delta_Hh/z$ gives
\[
A_H^2\equiv(c_H+2\delta_H+\delta_H^2/c_H)H \pmod{\mathfrak{p}}.
\]
Both sides have degree less than $d$, so this is an equality in $A$.
Since $H$ is an irreducible polynomial, one has $A_H=0$ and $\delta_H=-c_H$.
It follows that the twist $\phi'$ has good supersingular reduction at $(H)$.  
\end{proof}

\section{The case of even-degree prime levels}\label{sec:counting}
We retain the notation for $\phi$, $C$, $\lambda$, $\kappa$, and $k$
from Section~\ref{sec:characters}. 

\begin{definition}
Let $\nu$ be a positive integer with $q-1 \mid \nu$ and $\lambda^{\nu}\in\{1,\kappa^{\nu}\}$. 
We say that a polynomial $H \in A$ satisfies the condition (C$_{\mathfrak{p},  \nu}$) if $H$ is monic and there exist $v\in(A/\mathfrak{p})^\times$, $\delta_H\in\mathbf F_q^\times$, and  $A_H \in A$ such that 
\[
H\equiv vA_H-\delta_Hv^2\pmod{\mathfrak{p}},
\qquad v^{\nu}=1,\qquad  \deg A_H\le\lfloor \deg H/2\rfloor.
\]
\end{definition}

\begin{lemma}\label{lem:potent-good-lem}
Let $\nu$ be a positive integer with $q-1 \mid \nu$ and $\lambda^{\nu}\in\{1,\kappa^{\nu}\}$. 
Let $H \in A$ be a monic irreducible polynomial with  $\deg(H) < d$. 
If $\phi$ has potentially good reduction at $(H)$,  then $H$ satisfies the condition (C$_{\mathfrak{p},  \nu}$). 
\end{lemma}
\begin{proof}
Since $\phi$ has good reduction over a totally ramified extension of $F_{(H)}$ (see \cite[Lemma~2.3]{Ish22}),  
the Frobenius polynomial of the Drinfeld module obtained by reducing a model of $\phi$ with good reduction over
this totally ramified extension is $Z^2-A_HZ+\delta_HH$ with 
$\delta_H\in\mathbf F_q^\times$ and  $A_H \in A$ with $\deg A_H\le\lfloor \deg H/2\rfloor$ by \cite[Theorem 1.8]{Gek08}. 
Since $\lambda^{\nu}\in\{1,\kappa^{\nu}\}$, one root $z\in(A/\mathfrak p)^\times$ of $Z^2-A_HZ+\delta_HH \equiv 0 \pmod{\mathfrak{p}}$ satisfies $z^\nu \in \{1, h^{\nu}\}$, where $h := H \bmod{\mathfrak{p}}$. 
If $z^\nu = h^\nu$, then the other root $\delta_Hh/z$ satisfies $(\delta_Hh/z)^\nu = 1$. 
Hence we may assume that $z^\nu = 1$. 
We then put $v := z/\delta_H$. It follows that $v^{\nu}=1$ and $H\equiv vA_H-\delta_Hv^2\pmod{\mathfrak{p}}$. 
\end{proof}

Note that it follows from  Lemma \ref{lem:power}(1) that one can take $\nu = q^2-1$ if $d$ is even. 
For any even integer $d$,  we define
\[
\mathcal{S}_{\mathfrak{p}} := \{H\in A \mid \text{$\deg H=d-1$, $H$ is irreducible and satisfies (C$_{\mathfrak{p}, q^2-1}$)}\}
\]

\begin{lemma}\label{lem:count}
For any even integer $d\ge4$, one has  $\#\mathcal{S}_{\mathfrak{p}}\le(q^2-q)q^{(d-2)/2}$. 
\end{lemma}
\begin{proof}
    For $f\in A$, write $[T^i](f)$ for the coefficient of $T^i$ in $f$. 
We denote by $\remn_{\mathfrak{p}} \colon A/\mathfrak{p} \hookrightarrow A$ the $\mathbf{F}_q$-linear map sending each residue class to its unique  representative of degree less than $d$. 
Put  $h :=\lfloor (d-1)/2\rfloor = (d-2)/2$. 
Take an element $v \in \mathbf{F}_{q^2}^\times \subset (A/\mathfrak{p})^\times$. 
Let 
\[
\ell_v \colon \{A_H\in A \mid \deg A_H\le h\} \to \mathbf{F}_q; \qquad A_H \mapsto  [T^{d-1}]\remn_{\mathfrak{p}}(vA_H)
\]
be the $\mathbf{F}_q$-linear map. 
Since $v\in\mathbf F_{q^2}$, we have $v^2=\operatorname{Tr}_{\mathbf F_{q^2}/\mathbf F_q}(v)v-\operatorname{Norm}_{\mathbf F_{q^2}/\mathbf F_q}(v)$. 
Thus, for $H\in\mathcal{S}_{\mathfrak{p}}$, the congruence
$H\equiv vA_H-\delta_Hv^2\pmod{\mathfrak{p}}$ implies
\[
1=\ell_v(A_H)
-\delta_H\operatorname{Tr}_{\mathbf F_{q^2}/\mathbf F_q}(v)\ell_v(1).
\]
If $\ell_v\ne0$, this equation has $q^h$ solutions for $A_H$.
If $\ell_v=0$, it reads $1=0$ and has no solution.
Consequently, for each pair $(v,\delta_H)$ there are at most $q^h$
possible polynomials $A_H$.

For $u\in\mathbf F_q^\times$, the substitutions $(v,\delta_H,A_H)\mapsto(uv,u^{-2}\delta_H,u^{-1}A_H)$ satisfy
\[
\remn_{\mathfrak{p}}
\bigl((uv)(u^{-1}A_H)-(u^{-2}\delta_H)(uv)^2\bigr)
=
\remn_{\mathfrak{p}}(vA_H-\delta_Hv^2)
\]
and act freely on the pairs $(v,\delta_H)$.
Since $\ell_v=0$  if $v \in \mathbf{F}_q^{\times}$, one has 
\[
\#\mathcal{S}_{\mathfrak{p}} \leq  \frac{(q^2-1-(q-1))(q-1)}{q-1} q^{h} = (q^2-q)q^{h}. 
\]
\end{proof}

Let $\pi_q(e)$ denote the number of monic irreducible polynomials of
degree $e$ over $\mathbf F_q$.
Since $e\pi_q(e)=\sum_{d\mid e}\mu(d)q^{e/d}$, where $\mu$ is the M\"obius function, we obtain the inequality 
\begin{align}\label{pi-lower}
e\pi_q(e)
\ge q^e-\sum_{d\mid e, \, d\ge2}q^{e/d}
\ge q^e-\sum_{j=1}^{\lfloor e/2\rfloor}q^j = q^e- \frac{q^{\lfloor e/2\rfloor+1}-q}{q-1}.     
\end{align}

\begin{theorem}\label{thm:even}
Let $\mathfrak{p}$ be a prime ideal of $A$. 
If $d = \deg(\mathfrak{p})$ is even and $d \geq 4$, then $Y_0(\mathfrak{p})(F) = \emptyset$. 
\end{theorem}
\begin{proof} 
When $d=4$, Ishii proved in  \cite[Corollary~5.3]{Ish22} that    $Y_0(\mathfrak{p})(F) = \emptyset$. 
Hence we may assume that  $d\ge6$. 
We also note that $q \geq 3$. 

Suppose that $Y_0(\mathfrak p)(F)\ne\emptyset$, and choose a pair
$(\phi,C)$ over $F$ by Lemma~\ref{lem:descent}(1).
By Lemmas~\ref{lem:power}(1) and \ref{lem:potent-good-lem},
every irreducible polynomial $H$ of degree $d-1$ outside $\mathcal{S}_{\mathfrak{p}}$ satisfies $v_{(H)}(j(\phi))<0$. 
Since  $P \not\in \mathcal{S}_{\mathfrak{p}}$, it follows from Theorem~\ref{thm:bad-reduction-bound} that if 
\begin{equation}\label{eq:even-necessary}
(d-1)\bigl(\pi_q(d-1)-\#\mathcal{S}_{\mathfrak{p}}\bigr)> g_{\mathfrak p} = \frac{q^{d}-q^2}{q^2-1},
\end{equation}
then $Y_0(\mathfrak{p})(F) = \emptyset$. 

Let us prove the inequality \eqref{eq:even-necessary}. 
Write $d = 2r$, and first suppose that $d \geq 8$.   
 Lemma~\ref{lem:count} and the inequality \eqref{pi-lower} give
\[
(d-1)(\pi_q(d-1)-\#\mathcal{S}_{\mathfrak{p}}) \ge q^{2r-1}-\frac{q^r-q}{q-1}-(2r-1)(q-1)q^r.
\]
Since $q^{2r-1} - g_{\mathfrak{p}} > \left(1- q/(q^2-1)\right)q^{2r-1} \ge 5q^{2r-1}/8$, one has 
\[
(d-1)\bigl(\pi_q(d-1)-\#\mathcal{S}_{\mathfrak{p}}\bigr)-g_{\mathfrak p}
>q^r\left(\frac58q^{r-1}-(2r-1)(q-1)-\frac1{q-1}\right). 
\]
The expression in parentheses is positive by induction on $r$. 
Indeed, when $r = 4$, it follows that
\[
\frac{5}{8}q^{3} - 7(q-1) - \frac{1}{q-1} > 0
\]
for any $q \geq 3$. 
Hence the induction hypothesis shows that 
\begin{align*}
\frac{5}{8}q^{r} - (2r+1)(q-1) - \frac1{q-1} 
&> q\left((2r-1)(q-1) + \frac1{q-1}\right) - (2r+1)(q-1) - \frac1{q-1}
\\
&= (q(2r-1) - (2r+1))(q-1) + 1 > 0. 
\end{align*}
For $d=6$ and $q\ge7$, one has $g_{\mathfrak p}=q^4+q^2$ and
\[
5(\pi_q(5)-\#\mathcal{S}_{\mathfrak{p}})-g_{\mathfrak p}
\ge q^5-q-5(q^2-q)q^2-q^4-q^2 > 3q^3>0.
\]
The remaining cases are $(q,d) \in \{(3,6), (4,6), (5,6)\}$. 
For these three pairs, a computer-assisted calculation gives the following table:
\[
\begin{array}{c|r|r|r|r}
q & \pi_q(5) & \max_{\mathfrak{p}}\#\mathcal{S}_{\mathfrak{p}} & g_{\mathfrak p}
& 5(\pi_q(5)-\max_{\mathfrak{p}}\#\mathcal{S}_{\mathfrak{p}})-g_{\mathfrak p}\\
\hline
3 & 48  & 18  & 90  & 60\\
4 & 204 & 54  & 272 & 478\\
5 & 624 & 136 & 650 & 1790
\end{array}
\]
Here, the maximum is taken over all prime ideals $\mathfrak p$ of degree $6$. 
In particular, inequality~\eqref{eq:even-necessary} holds for each of these three pairs. 
\end{proof}

\section{The case of odd-degree prime levels}\label{sec:proofs}

Let $\phi$ be a rank-two Drinfeld module over $F$  with an $F$-rational cyclic $\mathfrak{p}$-subgroup $C$.

\begin{proposition}\label{prop:exclude-powers}
Suppose that $d$ is odd.
The isogeny character $\lambda$ of $\phi$ satisfies 
$\lambda^{q-1}\not\in\{1,\kappa^{q-1}\}$.
\end{proposition}
\begin{proof}
Suppose that $\lambda^{q-1}\in\{1,\kappa^{q-1}\}$.
Then it follows from Lemma~\ref{lem:potent-good-lem} that, for any prime $(H)$  with $1 \leq \deg H<d$ at which $\phi$ has potentially good reduction,
there exist $v\in(A/\mathfrak{p})^\times$, $\delta_H\in\mathbf F_q^\times$, and  $A_H \in A$ such that
\[
H\equiv vA_H-\delta_Hv^2\pmod{\mathfrak{p}},
\qquad v^{q-1}=1,\qquad  \deg A_H\le\lfloor \deg H/2\rfloor.
\]
Since $v \in \mathbf{F}_q$, one has $\deg(vA_H-\delta_Hv^2) < d$.
Hence $H = vA_H-\delta_Hv^2$. This contradicts the fact that $\deg A_H\le\lfloor \deg H/2\rfloor$.
Thus $\phi$ does not have potentially good reduction at every prime of degree $d-1$.
Write $d := 2r+1$. By the inequality \eqref{pi-lower}, one has
\[
(d-1)\pi_q(d-1)
\ge q^{2r}-\frac{q^{r+1}-q}{q-1} \geq \frac{q^{2r}}{2}
>\frac{q^{2r+1}-q}{q^2-1}=g_{\mathfrak p}
\]
since $q \geq 3$ and $r \geq 1$.
This contradicts Theorem~\ref{thm:bad-reduction-bound}.
\end{proof}

\begin{proposition}\label{prop:quadratic-character}
Let $k \in \mathbf{Z}$ be the exponent in Lemma~\ref{lem:Ishii} associated
with the $F$-rational cyclic $\mathfrak p$-subgroup $C$.
Suppose that  $q \equiv d \equiv 1 \pmod{2}$ and 
$2k\equiv1\pmod{(q^d-1)/(q-1)}$.
Then  one has $(q,d)=(3,3)$.
\end{proposition}
\begin{proof}
Write $2k = 1+r(q^d-1)/(q-1)$, where $r$ is odd.
Let $H \in A$ be a monic irreducible polynomial  of even degree $e<d$.
Let $\phi'$ be the twist of $\phi$ defined in Proposition~\ref{prop:supersingular}.
Then it follows that $\phi'$ has good reduction at $(H)$.
Its reduction over $A/(H)$ has Frobenius polynomial
\[
Z^2-c_HH,\qquad
c_H=\operatorname{Norm}_{(A/\mathfrak{p})/\mathbf F_q}(H)^r.
\]
This polynomial is irreducible over $F_\infty$ by
\cite[Theorem 1.8]{Gek08}.
Since $q$ is odd and $e$ is even, Hensel's lemma gives
\[
H\in T^e\bigl(1+T^{-1}\mathbf F_q[[T^{-1}]]\bigr)
\subset(F_\infty^\times)^2.
\]
Thus $c_H\notin  (\mathbf F_q^\times)^2$.
As $r$ is odd, we also have $\operatorname{Norm}_{(A/\mathfrak{p})/\mathbf F_q}(H) \notin  (\mathbf F_q^\times)^2$. 
Since $e$ is even, quadratic reciprocity gives
\[
\left(\frac{P}{H}\right) = \left(\frac{H}{P}\right) = \operatorname{Norm}_{(A/\mathfrak{p})/\mathbf F_q}(H)^{(q-1)/2} = -1.
\]
Let $C_P/\mathbf F_q$ denote the smooth projective curve with affine equation $y^2=P(x)$.
Take an even integer $n$ with $0<n<d$.
Since $d$ is odd, one has $\#C_P(\mathbf F_{q^n})=1 +
\#\{(x,y)\in\mathbf{F}_{q^n}^2 \mid y^2=P(x)\}$.
For $x\in\mathbf F_{q^{n/2}}$, one has
$P(x)\ne0$ since $n<d$.
Hence $P(x)$ is a square in $\mathbf{F}_{q^n}$. Thus there are exactly two points above $x$.
If $x\in\mathbf F_{q^n}\setminus\mathbf F_{q^{n/2}}$, let $H_x$
be the minimal polynomial of $x$ over $\mathbf F_q$, and put $e:=\deg H_x$.
Then $e\mid n$ and $e\nmid n/2$. Hence $e$ is even and $n/e$ is odd.
In this case, $\left(\frac{P}{H_x}\right)=-1$.
Thus $P(x)$ is a nonsquare in $\mathbf F_{q^e}$ and remains a nonsquare
in the odd-degree extension $\mathbf F_{q^n}$, so there are no points above $x$.
Thus we obtain
\[
\#C_P(\mathbf F_{q^n}) = 1 + 2q^{n/2}
\]
for any even $n$ with $0<n<d$.
The curve $C_P$ has genus $(d-1)/2$ and one rational point at infinity.
The Hasse--Weil bound for $C_P(\mathbf{F}_{q^{d-1}})$  gives
\[
q^{(d-1)/2}\le d+1.
\]
This implies $(q,d)=(3,3)$.
\end{proof}

\begin{theorem}\label{thm:main-y0}
Let $\mathfrak{p}$ be a prime ideal of $A$.
If $d = \deg(\mathfrak{p})$ is odd and $d>q$, then $Y_0(\mathfrak p)(F)=\emptyset$.
\end{theorem}
\begin{proof}
Suppose that $Y_0(\mathfrak p)(F)\ne\emptyset$, and choose
$(\phi,C)$ by Lemma~\ref{lem:descent}(1).
Lemma~\ref{lem:power}, together with
Proposition~\ref{prop:exclude-powers}, shows that $v_{\mathfrak p}(j(\phi))\ge0$ and $v_{(T-c)}(j(\phi))\ge0$ for all $c\in\mathbf F_q$.
It then follows from Lemma~\ref{lem:linear-good-reduction} that $q$ is odd and $2k\equiv1\pmod{(q^d-1)/(q-1)}$.
Proposition~\ref{prop:quadratic-character} forces $(q,d)=(3,3)$.
This contradicts the assumption $d>q$.
\end{proof}

\begin{remark}
When $q=3$ and $\mathfrak p=(P)$ with $P(T)=T^3-T\pm 1$, the Drinfeld modular curve $Y_0(\mathfrak p)$ admits an $F$-rational CM point.  
More strongly, for any nontrivial ideal $\mathfrak n$ of $A$ with $\deg(\mathfrak n)\geq 3$, the Drinfeld modular curve $Y_0(\mathfrak n)$ admits an $F$-rational CM point only in these two cases (see \cite[Remark~4.6]{Sch03}). 
\end{remark}

\section{Rational points on \texorpdfstring{$Y_1(\mathfrak p)$}{Y1(p)}}\label{sec:torsion}

\begin{theorem}\label{thm:rational-torsion}
Let $\mathfrak{p}$ be a prime ideal of $A$.
If $d = \deg(\mathfrak{p}) \ge3$, then $Y_1(\mathfrak p)(F)=\emptyset$.
\end{theorem}
\begin{proof}
When $d$ is even, the assertion follows from Theorem~\ref{thm:even}.
Suppose that $d$ is odd and $Y_1(\mathfrak p)(F)\ne\emptyset$.
Then  there are a rank-two Drinfeld
$A$-module $\phi/F$ and a point $a\in\phi[\mathfrak p](F)\setminus\{0\}$.
The isogeny character associated with $Aa$ is trivial, so $\lambda=1$.
This contradicts Proposition~\ref{prop:exclude-powers}.
\end{proof}

\bibliographystyle{alpha}
\bibliography{main}

\begin{thebibliography}{Tam95}

\bibitem[Arm12]{Arm12}
C\'ecile Armana.
\newblock Torsion des modules de {D}rinfeld de rang 2 et formes modulaires de
  {D}rinfeld.
\newblock {\em Algebra Number Theory}, 6(6):1239--1288, 2012.

\bibitem[CT21]{CT21}
Brian Conrad and Michael Temkin.
\newblock Descent for nonarchimedean analytic spaces.
\newblock {\em Tunis. J. Math.}, 3(4):689--748, 2021.

\bibitem[Dri74]{Dri74}
V.~G. Drinfel'd.
\newblock Elliptic modules.
\newblock {\em Mat. Sb. (N.S.)}, 94(136):594--627, 656, 1974.

\bibitem[Gek86]{Gek86}
Ernst-Ulrich Gekeler.
\newblock {\"U}ber {D}rinfeldsche {M}odulkurven vom {H}ecke-{T}yp.
\newblock {\em Compositio Math.}, 57(2):219--236, 1986.

\bibitem[Gek08]{Gek08}
Ernst-Ulrich Gekeler.
\newblock Frobenius distributions of {D}rinfeld modules over finite fields.
\newblock {\em Trans. Amer. Math. Soc.}, 360(4):1695--1721, 2008.

\bibitem[Ish22]{Ish22}
Shun Ishii.
\newblock On isogeny characters of {D}rinfeld modules of rank two.
\newblock {\em Math. Z.}, 301(1):455--470, 2022.

\bibitem[Ish26]{IshiiNote}
Shun Ishii.
\newblock On certain bases of finite fields and their application to rational
  points on {Drinfeld} modular curves, August 2026.
\newblock Unpublished note, dated August 15, 2026.

\bibitem[Maz77]{Maz77}
B.~Mazur.
\newblock Modular curves and the {E}isenstein ideal.
\newblock {\em Inst. Hautes \'Etudes Sci. Publ. Math.}, 47:33--186, 1977.
\newblock With an appendix by Mazur and M. Rapoport.

\bibitem[Maz78]{Maz78}
B.~Mazur.
\newblock Rational isogenies of prime degree (with an appendix by {D}.
  {G}oldfeld).
\newblock {\em Invent. Math.}, 44(2):129--162, 1978.

\bibitem[P{\'a}l07]{Pal07}
Ambrus P{\'a}l.
\newblock On the {E}isenstein ideal of {D}rinfeld modular curves.
\newblock {\em Int. J. Number Theory}, 3(4):557--598, 2007.

\bibitem[P{\'a}l10]{Pal10}
Ambrus P{\'a}l.
\newblock On the torsion of {D}rinfeld modules of rank two.
\newblock {\em J. Reine Angew. Math.}, 640:1--45, 2010.

\bibitem[Pap04]{Pap04}
Mihran Papikian.
\newblock On component groups of {J}acobians of {D}rinfeld modular curves.
\newblock {\em Ann. Inst. Fourier (Grenoble)}, 54(7):2163--2199, 2004.

\bibitem[Pap07]{Pap07}
Mihran Papikian.
\newblock Abelian subvarieties of {D}rinfeld {J}acobians and congruences modulo
  the characteristic.
\newblock {\em Math. Ann.}, 337(1):139--157, 2007.

\bibitem[Poo97]{Poo97}
Bjorn Poonen.
\newblock Torsion in rank {$1$} {D}rinfeld modules and the uniform boundedness
  conjecture.
\newblock {\em Math. Ann.}, 308(4):571--586, 1997.

\bibitem[PT21]{PT20I}
J\'er\^ome Poineau and Daniele Turchetti.
\newblock Berkovich curves and {S}chottky uniformization {I}: the {B}erkovich
  affine line.
\newblock In {\em Arithmetic and geometry over local fields---{VIASM} 2018},
  volume 2275 of {\em Lecture Notes in Math.}, pages 179--223. Springer, Cham,
  [2021] \copyright 2021.

\bibitem[PW15]{PW15}
Mihran Papikian and Fu-Tsun Wei.
\newblock The {E}isenstein ideal and {J}acquet-{L}anglands isogeny over
  function fields.
\newblock {\em Doc. Math.}, 20:551--629, 2015.

\bibitem[R{\"o}s15]{Ros15}
Damian R{\"o}ssler.
\newblock On the group of purely inseparable points of an abelian variety
  defined over a function field of positive characteristic.
\newblock {\em Comment. Math. Helv.}, 90(1):23--32, 2015.

\bibitem[Sch03]{Sch03}
Andreas Schweizer.
\newblock On the uniform boundedness conjecture for {D}rinfeld modules.
\newblock {\em Math. Z.}, 244(3):601--614, 2003.

\bibitem[Tam95]{Tam95}
Akio Tamagawa.
\newblock The {E}isenstein quotient of the {J}acobian variety of a
  {D}rinfel'd{} modular curve.
\newblock {\em Publ. Res. Inst. Math. Sci.}, 31(2):203--246, 1995.

\bibitem[Tei91]{Tei91}
Jeremy~T. Teitelbaum.
\newblock The {P}oisson kernel for {D}rinfel'd{} modular curves.
\newblock {\em J. Amer. Math. Soc.}, 4(3):491--511, 1991.

\end{thebibliography}

\end{document}